\documentclass[a4paper, 10pt]{article}

\usepackage[latin1]{inputenc}
\usepackage[T1]{fontenc}

\usepackage[dvips]{graphicx}
\usepackage{color}

\usepackage{amsthm}  
\usepackage{amsmath} 
\usepackage{amssymb} 
\usepackage{latexsym} 

\theoremstyle{plain}
\newtheorem{theorem}{Theorem}

\newtheorem{proposition}[theorem]{Proposition}
\newtheorem{observation}[theorem]{Observation}
\newtheorem{corollary}[theorem]{Corollary}

\newtheorem{lemma}[theorem]{Lemma}

\newcommand*{\col}{\mathrm{col}}

\newcommand*{\sym}{\mathrm{S}}

\newcommand{\mc}[1]{\mathcal{ #1 }}

\usepackage{amsmath}
\usepackage{amssymb}
\usepackage{latexsym} 
\usepackage{amsthm}

\usepackage{geometry}

\usepackage{graphicx,tikz}
\usepackage{subfigure}

\def\ldotsplus{\mathinner{\ldotp\ldotp\ldotp\ldotp}}
\def\fourdots{\relax\ifmmode\ldotsplus\else$\m@th \ldotsplus\,$\fi}

\usepackage[numbers]{natbib}

\author{Matthias Kriesell \& Anders Sune Pedersen\\
\small  Dept. of Mathematics and Computer Science\\[-0.8ex]
\small  University of Southern Denmark\\[-0.8ex]
\small  Campusvej 55, 5230 Odense M, Denmark\\
\small \texttt{\{asp, kriesell\}@imada.sdu.dk}
}

\title{On graphs double-critical with respect to the~colouring~number}

\begin{document}
\maketitle

\begin{abstract}
The \emph{colouring number $\col(G)$} of a graph $G$ is the smallest
integer $k$ for which there is an ordering of the vertices of $G$ such
that when removing the vertices of $G$ in the specified order no
vertex of degree more than $k-1$ in the remaining graph is removed at
any step. An edge $e$ of a graph $G$ is said to be
\emph{double-$\col$-critical} if the colouring number of $G-V(e)$ is
at most the colouring number of $G$ minus $2$. A connected graph $G$
is said to be \emph{double-$\col$-critical} if each edge of $G$ is
double-$\col$-critical. We characterise the double-$\col$-critical
graphs with colouring number at most $5$. In addition, we prove that
every $4$-$\col$-critical non-complete graph has at most half of its
edges being double-$\col$-critical, and that the extremal graphs are
precisely the odd wheels on at least six vertices. We observe that for
any integer $k$ greater than $4$ and any positive number $\epsilon$,
there is a $k$-$\col$-critical graph with the ratio of
double-$\col$-critical edges between $1- \epsilon$ and $1$.
\end{abstract}

\section{Introduction}
All graphs considered in this paper are assumed to be simple and
finite.\footnote{The reader is referred to~\cite{BondyAndMurty2008}
  for definitions of graph-theoretic concepts used but not explicitly
  defined in this paper.} The cycle on $n$ vertices is denoted by
$C_n$. The complete graph $K_n$ on $n$ vertices is referred to as an
\emph{$n$-clique}. Let $G$ denote a graph. The number of vertices in a
largest clique contained in $G$ is denoted by $\omega(G)$. The
vertex-connectivity of $G$ is denoted by $\kappa(G)$. The number of
vertices and edges in $G$ is denoted by $n(G)$ and $m(G)$,
respectively.

 Given a vertex $v$ in $G$, $N(v,G)$ denotes the set of vertices in
 $G$ adjacent to $v$; $\deg(v, G)$ denotes the cardinality of
 $N(v,G)$, and it is referred to as the \emph{degree} of $v$ (in
 $G$). A vertex of degree $1$ is referred to as a \emph{leaf}. Given a
   subset $S$ of the vertices of $G$, the subgraph of $G$ induced by
   the vertices of $S$ is denoted by $G[S]$, and we let $N(S,G)$
   denote the set $\cup_{s \in S} N(s,G) \setminus S$.  The
   \emph{square of a graph $G$}, denoted by $G^2$, is the graph
   obtained from $G$ by adding edges between any pair of vertices of
   $G$ which are at distance $2$ in $G$. Given two graphs $H$ and $G$,
   the \emph{complete join} of $G$ and $H$, denoted by $G + H$, is the
   graph obtained from two disjoint copies of $H$ and $G$ by joining
   each vertex of the copy of $G$ to each vertex of the copy of
   $H$. The chromatic number of $G$ is denoted by $\chi(G)$, while the
   list-chromatic number of $G$ is denoted by $\chi_\ell(G)$. Let
   $\psi$ denote some graph parameter. An edge $e$ of $G$ is said to
   be \emph{double-$\psi$-critical} if $\psi(G - V(e)) \leq \psi(G) -
   2$. A connected graph $G$ is said to be
   \emph{double-$\psi$-critical} if each edge of $G$ is
   double-$\psi$-critical. For brevity, we may also refer to
   double-$\chi$-critical edges and graphs as, simply,
   \emph{double-critical} edges and graphs, respectively. \\

The introduction of the concept of double-$\psi$-critical graphs
in~\cite{ASP_PhD_Thesis} was inspired by a special case of the
Erd\H{o}s-Lov\'asz Tihany Conjecture~\cite{TihanyProblem2}, namely the
special case which states that the complete graphs are the only
double-critical graphs. We refer to this special case of the
Erd\H{o}s-Lov\'asz Tihany Conjecture as the \emph{Double-Critical
  Graph Conjecture}. The Double-Critical Graph Conjecture is settled
in the affirmative for the class of graphs with chromatic number
at most $5$, but remains unsettled for the class of graphs with chromatic number at least $6$~\cite{JensenToft95, MR995391, MR882614, MR1221590}.

In~\cite{ASP_PhD_Thesis}, it was proved that if $G$ is a
double-$\chi_\ell$-critical graph with $\chi_\ell(G) \leq 4$, then $G$
is complete. It is an open problem whether there is a non-complete
double-$\chi_\ell$-critical graph with list-chromatic number at least
$5$.

The double-$\kappa$-critical graphs, which in the literature are
referred to as \emph{contraction-critical} graphs (since the
vertex-connectivity drops by one after contraction of any edge), are
well-understood in the case where $\kappa$ is $4$. Some structural
results have been obtained for contraction-critical graphs with
vertex-connectivity $5$. (See \cite[Sec. 4]{MR1892444} for references
on contraction-critical graphs.)

Bjarne Toft\footnote{Private communication to the second author from
  Bjarne Toft, Odense, August 2008.} posed the problem of
characterising the double-$\col$-critical graphs. Here $\col$ denotes
the colouring number which is defined in the paragraph below.

In this paper, we characterise the double-$\col$-critical graphs with
colouring number at most $5$.

In the remaining part of this section, we define the colouring
number and present some fundamental properties of this graph
parameter.
\paragraph{The colouring number of a graph.} Suppose that we are given a non-empty graph $G$ and an ordering $v_1,
 \ldots, v_n$ of the vertices of $G$\label{page:ErdosHajnalCol}. Now
 we may colour the vertices of $G$ in the order $v_1, \ldots, v_n$
 such that in the $i$th step the vertex $v_i$ is assigned the smallest
 possible positive integer which is not assigned to any neighbour of
 $v_i$ among $v_1, \dots, v_{i-1}$. This produces a colouring of $G$
 using at most
$$\max_{i \in \{1, \dots, n \}} \deg(v_i, G[v_1, \ldots, v_i])  + 1$$
colours. Taking the minimum over the set $\sym_n$ of all permutations of
$\{1, \dots, n\}$, we find that the chromatic number of $G$ is at most
\begin{equation} \label{eq:987qwenlasdf923}
\min_{\pi \in \sym_n} \left\{ \max_{i \in \{1, \dots, n\}} \deg(v_{\pi(i)}, G[v_{\pi(1)}, \ldots, v_{\pi(i)}])  \right\} + 1
\end{equation}
The number in~\eqref{eq:987qwenlasdf923} is called the \emph{colouring
  number} of $G$, and it is denoted by $\col(G)$. The colouring number
of the empty graph $K_0$ is defined to be
zero. By~\eqref{eq:987qwenlasdf923}, $\col(G) \leq \Delta(G)+1$ for
any graph $G$. The colouring number was introduced by Erd\H{o}s and
Hajnal~\cite{MR0193025}, but equivalent concepts were introduced
independently by several other authors. It can be shown (see, for
instance,~\cite{MR1373659}) that the colouring number of any non-empty
graph $G$ is equal to
\begin{equation} \label{eq:u8934jkasf3498} \max \{ \delta (H) \mid H
  \textnormal{ is an induced subgraph of } G \} + 1
\end{equation}
and that the colouring number can be computed in polynomial
time~\cite{MR709826}. The non-empty graphs with colouring number at
most $k+1$ are also said to be
\emph{$k$-degenerate}~\cite{MR0266812}. Thus, a non-empty graph $G$ is
$k$-degenerate if and only if there is an ordering of the vertices of
$G$ such that when removing the vertices of $G$ in the specified order
no vertex of degree more than $k$ in the remaining graph is removed at
any step. We may think of a $k$-degenerate graph as a graph that can
be `degenerated' to the empty graph by removing vertices of degree at
most $k$.

The colouring number is monotone on subgraphs, that is, if $F$ is a
subgraph of a graph $G$ then $\col(F) \leq \col(G)$. For ease of
reference, we state the following elementary facts concerning the
colouring number of graphs.

\begin{observation} \label{obs:3457089asdflkjh345} For any graph $G$,
\begin{itemize} 
\item[\emph{(i)}] $\col(G)=0$ if and only if $G$ is the empty graph, 
\item[\emph{(ii)}] $\col(G)=1$ if and only if $G$ contains at least one vertex but no edges,
\item[\emph{(iii)}] $\col(G)=2$ if and only if $G$ is forest containing at
  least one edge, and
\item[\emph{(iv)}] $\col(G)\geq 3$ if and only if $G$ contains at least one cycle.
\end{itemize}
\end{observation}

 A graph $G$ is said to be \emph{$k$-$\col$-critical}, or, simply,
 \emph{$\col$-critical}, if $\col (G) = k$ and $\col (F) < k$ for
 every proper subgraph $F$ of $G$. Similarly, a graph $G$ is said to
 be \emph{$k$-$\col$-vertex-critical}, or, simply,
 \emph{$\col$-vertex-critical}, if $\col(G) = k$ and $\col(F) < k$ for
 every induced proper subgraph $F$ of $G$. It is easy to see that
 every connected $r$-regular graph is $(r+1)$-$\col$-critical.

\begin{observation} \label{obs:elementaryColCritical} For any
  $\col$-vertex-critical graph $G$,
\begin{itemize} 
\item[\emph{(i)}] $\col(G)=0$ if and only if $G \simeq K_0$,
\item[\emph{(ii)}] $\col(G)=1$ if and only if $G \simeq K_1$,
\item[\emph{(iii)}] $\col(G)=2$ if and only if $G \simeq K_2$, and
\item[\emph{(iv)}] $\col(G)=3$ if and only if $G$ is a cycle.
\end{itemize}
\end{observation}

\begin{observation} \label{obs:ColDropsByOneExactly}
  For any graph $G$ and any element $x \in E(G) \cup V(G)$, if $\col(G - x) <
  \col(G)$ then $\col(G-x) = \col(G) - 1$.
\end{observation}

\begin{observation} \label{obs:395afhq9348alsdif45ur}
A graph $G$ is $\col$-vertex-critical if and only if $\col(G - v) < \col(G)$ for every vertex $v$ in $G$.
\end{observation}
\begin{observation} \label{obs:existenceOfColCriticalSubgraph}
  Given any graph $G$, there is a $\col$-critical subgraph $F$ of $G$
  with $\col (G) = \col(F) = \delta(F)+1$. In particular, if $G$ is
  $\col$-critical then $\col(G) = \delta(G) + 1$.
\end{observation}
\begin{proof}
  Recall that $\col(G) = \max \{ \delta(H) \mid H \subseteq G \} +
  1$. Among the subgraphs $H$ of $G$ with $\col (G) = \delta(H) + 1$,
  let $F$ denote a minimal one, that is, $\delta (F') < \delta (F)$
  for every proper subgraph $F'$ of $F$. (This minimum exists since
  $G$ is finite.) Then $F$ is $\col$-critical with
  $\col(F) = \delta(F) + 1 = \col(G)$.
\end{proof}

\begin{observation}\label{obs:existenceOfColVertexCriticalSubgraph}
  Given any graph $G$, there is a $\col$-vertex-critical induced
  subgraph $F$ of $G$ with $\col (G) = \col(F) = \delta(F)+1$.  In
  particular, if $G$ is $\col$-vertex-critical then $\col(G) =
  \delta(G) + 1$.
\end{observation}
\begin{proof}
  Let $F$ denote a minimal induced subgraph of $G$ with $\col(F) =
  \col(G)$. This implies $\col(F') < \col(F)$ for any induced proper
  subgraph $F'$ of $F$, in particular, $F$ is a $\col$-vertex-critical
  graph. Suppose $\col(F) > \delta(F)+1$. Then there is some proper
  induced subgraph $F'$ of $F$ with $\delta(F') + 1 = \col(F)$, and so
  $\col(F') \geq \col(F)$, a contradiction. Hence $\col(F) = \delta(F)
  + 1$. If $G$ is $\col$-vertex-critical, then $F=G$, and the desired
  result follows.
\end{proof}

The two following results may be of interest in their own right.

\begin{proposition}[Pedersen~\cite{ASP_PhD_Thesis}] \label{prop:completeJoinsAndColouringNumber} For
  any two non-empty disjoint graphs $G_1$ and $G_2$, the colouring number of the
  complete join $G_1 + G_2$ is at most
\begin{equation} \label{eq:435897asdfjk}
\min \{ \col(G_1) + n(G_2), \col(G_2) + n(G_1) \} 
\end{equation}
and at least
\begin{equation} \label{eq:2345897sdfjk3498}
\min \{ \col(G_1) + n(J_2), \col(G_2) + n(J_1) \}
\end{equation}
where, for each $i \in \{ 1, 2 \}$, $J_i$ is any subgraph of $G_i$ with
minimum degree equal to $\col(G_i) - 1$. 

 If, in addition, $\col(G_i) = \delta(G_i) + 1$ for each $i \in \{1, 2 \}$
 (in particular, if both $G_1$ and $G_2$ are $\col$-vertex-critical),
 then the colouring number of the complete join $G_1 + G_2$ is equal
 to the minimum in~\eqref{eq:435897asdfjk}.
\end{proposition}
A graph $G$ is said to be \emph{decomposable} if there is a partition
of $V(G)$ into two (non-empty) sets $V_1$ and $V_2$ such that, in $G$,
every vertex of $V_1$ is adjacent to every vertex of $V_2$. Given any
graph $G$, we let $V_\delta(G)$ denote the set of vertices of $G$ of
minimum degree in $G$. Clearly, $V_\delta(G)$ is non-empty for any
non-empty graph.
\begin{proposition}[Pedersen~\cite{ASP_PhD_Thesis}]~\label{prop:completeJoinsAndColCritical} Let $G$ denote a decomposable graph. Then $G$ is
  $\col$-critical if and only if the vertex set of $G$ can be
  partitioned into two sets $V_1$ and $V_2$ such that $G = G_1 + G_2$,
  where $G_i := G[V_i]$ for $i \in \{1,2\}$, $G_1$ is regular, and
\begin{itemize}
\item[\emph{(i)}] $V(G_2) \setminus V_\delta(G_2)$ is an independent
  set of $G_2$, and 
$$\delta(G_1) + n(G_2) = \delta(G_2) + n(G_1)$$
or
\item[\emph{(ii)}] $G_2$ is an edgeless graph, and 
$$n(G_1) - \delta(G_1) - n(Q) < n(G_2) < n(G_1) - \delta(G_1)$$
where $Q$ denotes a smallest component of $G_1$ (in terms of the
number of vertices).
\end{itemize}
Moreover, $\col(G) = \delta(G_1) + n(G_2) + 1$ in both \emph{(i)} and \emph{(ii)}.
\end{proposition}

\section{Double-col-critical graphs}
The analogue of the Double-Critical Graph Conjecture with $\chi$
replaced by $\col$ does not hold. For instance, the non-complete graph
$C_6^2$ is $4$-regular, $5$-$\col$-critical, and
double-$\col$-critical. Since $C_6^2$ is planar, it also follows that
it is not even true that every double-$\col$-critical graph with
colouring number $5$ contains a $K_5$
minor. (In~\cite{KawarabayashiPedersenToftEJC2010}, it was proved that
every double-critical graph $G$ with $\chi(G) \leq 7$ at least
contains a $K_{\chi(G)}$ minor.) It is easy to see that the square of
any cycle of length at least $5$ is a double-$\col$-critical graph
with colouring number $5$.
\begin{observation} \label{obs:345770asdfs456dflkj345}
Any double-$\col$-critical graph is $\col$-vertex-critical.
\end{observation}
\begin{proof}
  Let $G$ denote a double-$\col$-critical graph. If there are no
  vertices in $G$, then we are done. Let $v$ denote an arbitrary but
  fixed vertex of $G$. If there is no vertex in $G$ adjacent to $v$,
  then we are done, since then, by the connectedness of $G$, $G$ is
  just the singleton $K_1$. Let $u$ denote a neighbour of $v$. By
  Observation~\ref{obs:395afhq9348alsdif45ur}, we need to show $\col(G
  - v) < \col(G)$. The fact that $G$ is double-$\col$-critical implies
  $\col(G - u - v) \leq \col(G) - 2$. Suppose $\col(G-v) \geq
  \col(G)$. Then
$$\col((G - v) - u) \leq \col(G) - 2 = \col(G-v) - 2$$
which contradictions Observation~\ref{obs:ColDropsByOneExactly}. This
shows $\col(G - v)$ is strictly less than $\col(G)$, as desired.
\end{proof}

\begin{observation} \label{obs:j3495yfh394tsdf}
For each integer $k \in \{0,1,2,3,4 \}$, the only double-$\col$-critical graph
with colouring number $k$ is the $k$-clique.
\end{observation}
\begin{proof}
  Let $G$ denote a double-$\col$-critical graph, and define $k :=
  \col(G)$. Then, by Observation~\ref{obs:345770asdfs456dflkj345}, $G$
  is also $\col$-vertex-critical, and so, by
  Observation~\ref{obs:existenceOfColVertexCriticalSubgraph},
  $\delta(G) = k - 1$. If $k \leq 3$, then the desired result follows
  immediately from
  Observation~\ref{obs:elementaryColCritical}. Suppose $k=4$. Then,
  for any edge $e \in E(G)$, $\col(G - V(e)) \leq 2$ and so, by
  Observation~\ref{obs:3457089asdflkjh345}, $G - V(e)$ is a
  forest. Fix an edge $xy \in E(G)$. If $G - x - y$ contains no edges,
  then $G$ is $2$-degenerate and so $\col(G) \leq 3$, a
  contradiction. Let $T$ denote a component of $G - x - y$ with at
  least one edge, and let $u$ and $v$ denote two leafs of $T$. Since,
  as noted above, $\delta(G) = 3$, it follows that both $u$ and $v$
  are adjacent to both $x$ and $y$. If $u$ and $v$ are adjacent in
  $T$, then $G[ \{u,v,x,y \} ] \simeq K_4$, and so, since $G$ is also
  $\col$-vertex-critical, $G \simeq K_4$. Hence we may assume that $u$
  has a neighbour in $T-v$. Now $G[ \{ x,y,v \}]$ is a $3$-clique in
  $G - t - u$ and so $\col(G - t - u) \geq \col (G[ \{ x,y,v \}] ) =
  3$, a contradiction. This completes the proof.
\end{proof}
\begin{figure}[htbp]
  \begin{center}
    \mbox{ \subfigure[The graph $Q_1$.]{\scalebox{0.6}{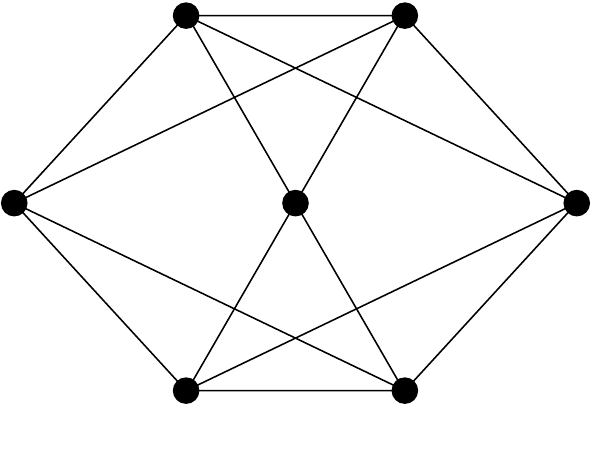}} \hspace{1cm}
      \subfigure[The graph $Q_2$.]{\scalebox{0.6}{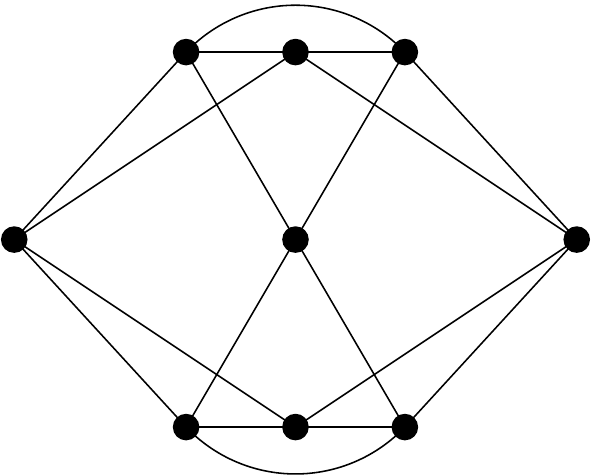}} 
 \hspace{1cm}
      \subfigure[The graph $Q_3$.]{\scalebox{0.6}{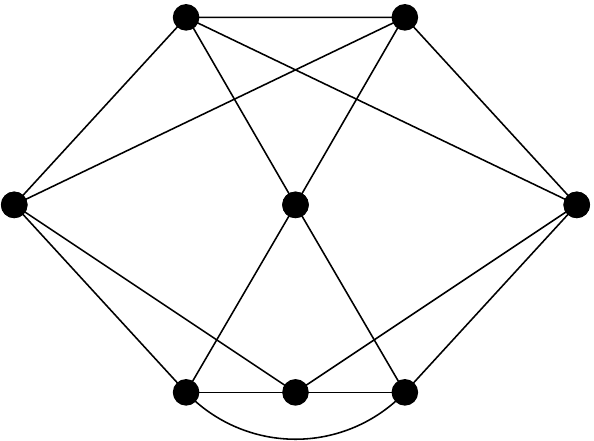}} }
 \end{center}
 \caption{The graphs $Q_1$, $Q_2$, and $Q_3$, depicted above, are the
   only double-$\col$-critical graphs with colouring number $5$ which
   are not squares of cycles.}
 \label{fig:nonSquareCycleDoubleColCritical}
\end{figure}

It is easy to verify that the graphs $Q_1$, $Q_2$, and $Q_3$ in
Figure~\ref{fig:nonSquareCycleDoubleColCritical} are
double-$\col$-critical and have colouring number $5$. None of the
graphs $Q_1$, $Q_2$, and $Q_3$ are squares of a cycle. We shall see
that $Q_1$, $Q_2$, $Q_3$, and the squares of the cycles of length at
least $5$ are all the double-$\col$-critical graphs with colouring
number $5$. First a few preliminary observations.
\begin{observation} \label{obs:23fds45as79} If $G$ is a
  double-$\col$-critical graph, then $\delta(G) = \col(G) - 1$ and
  every pair of adjacent vertices of $G$ has a common neighbour of
  degree $\delta(G)$ in $G$.
\end{observation}
\begin{proof}
  Let $G$ denote a double-$\col$-critical graph. Then, by
  Observation~\ref{obs:345770asdfs456dflkj345}, $G$ is also
  $\col$-vertex-critical, and so, by
  Observation~\ref{obs:existenceOfColVertexCriticalSubgraph},
  $\delta(G) = \col(G)-1$. Let $xy$ denote an arbitrary edge of
  $G$. Now, by the definition of the colouring number, $G-x-y$ has
  minimum degree at most $(\col(G) - 2) - 1$ which is equal to
  $\delta(G) - 2$. This means that some vertex of $V(G) \setminus \{
  x,y \}$, say $z$, which has degree at least $\delta(G)$ in $G$ has
  degree at most $\delta(G) - 2$ in $G-x-y$. The only way this can
  happen is if $z$ has degree $\delta(G)$ in $G$ and is adjacent to
  both $x$ and $y$ in $G$. This completes the argument.
\end{proof}
\begin{observation} \label{obs:345987asdlhf435asdf} If $G$ is a
  non-complete double-$\col$-critical graph, then $G$ does not contain
  a clique of order $\col(G) - 1$.
\end{observation}
\begin{proof}
  Let $G$ denote a non-complete double-$\col$-critical graph. By
  Observation~\ref{obs:345770asdfs456dflkj345}, $G$ is
  $\col$-vertex-critical, and so, since $G$ is also non-complete, $G$
  cannot contain a clique of order more than $\col(G) - 1$. Also, by
  Observation~\ref{obs:345770asdfs456dflkj345}, $\delta(G) = \col(G) -
  1$. Suppose that $G$ contains a clique $K$ of order $\col(G) -
  1$. Clearly, $G - V(K)$ is not empty. If $G - V(K)$ contains an
  edge $xy$, then $\col(G) - 1 = \col(K) \leq \col(G - x - y) \leq
  \col(G) - 2$, a contradiction. Hence $G-V(K)$ is edgeless, and so,
  since $\delta(G) = \col(G) - 1 = n(K)$, it follows that each vertex
  of $V(G)\setminus V(K)$ is adjacent to every vertex of $V(K)$, in
  particular, $G$ contains a clique of order $\col(G)$, a
  contradiction.
\end{proof}

\begin{proposition} \label{prop:98ysdfhlkjasdf3}
  Every double-$\col$-critical graph with colouring number at least $3$ is
  $2$-connected.
\end{proposition}
\begin{proof}
  Let $G$ denote a double-$\col$-critical graph with $\col(G) \geq
  3$. Since $G$ is $\col$-vertex-critical, it is connected with
  $\col(G) = \delta(G)+1$. Suppose $G$ is not $2$-connected, and let
  $x$ denote a cutvertex of $G$. Each component of $G-x$ has minimum
  degree at least $\delta(G) - 1$. Let $C$ denote a component of
  $G-x$, and let $e$ denote some edge of $G - V(C)$. Then, using the fact that $G$ is double-$\col$-critical, $\col$ is monotone, and $C$ is a subgraph of $G-V(e)$, we obtain
$$\col(G) - 1 = (\delta(G) + 1) - 1 \leq \col(C) \leq \col(G - V(e)) =
  \col(G) - 2$$
a contradiction. This shows that $G$ must be $2$-connected.
\end{proof}
\paragraph{Double-col-critical graphs with colouring number $\mathbf{5}$.}
\begin{observation} \label{obs:345897awlekrf} If $G$ is a
  double-$\col$-critical graph with colouring number $5$ and $ab \in
  E(G)$, then $a$ or $b$ has degree $4$ in $G$.
\end{observation} 
We shall say that a \emph{$k$-neighbour} of a vertex $x$ is a
neighbour of $x$ of degree $k$.
\begin{proof}[Proof of Observation~\ref{obs:345897awlekrf}.]
  Let $G$ denote a double-$\col$-critical graph with colouring number
  $5$. Then $\col(G) = \delta(G) + 1 = 5$. Let $ab$ denote an edge of
  $G$. Suppose that both $a$ and $b$ have degree greater than $4$ in
  $G$. By Observation~\ref{obs:23fds45as79}, there is a common
  $4$-neighbour $c$ of $a$ and $b$. We shall make repeated use of
  Observation~\ref{obs:23fds45as79} and
  Observation~\ref{obs:345987asdlhf435asdf}. The latter observation
  implies that $G$ contains no $4$-clique. There is a common
  $4$-neighbour $d$ of $a$ and $c$. Since $\omega(G) \leq 3$, $d$ is
  not adjacent to $b$. This implies that there is a common
  $4$-neighbour $e$ of $b$ and $c$ and $e$ is not identical to
  $d$. The vertex $e$ is not adjacent to $a$. The vertex $a$ has
  degree at least $5$, and so the common $4$-neighbour of $c$ and $d$
  must be $e$. We note that $\{ a,b,c,d,e \}$ induce a subgraph of $G$
  of minimum degree $3$. Hence $G - \{ a,b,c,d,e \}$ contains no
  edges. Moreover, $\col(G - \{c,d, e\}) \leq \col(G) - 2$, and so $G
  - \{ c,d,e \}$ contains a vertex $f$ of degree at most $2$. Since
  $a$ and $b$ both have degree at least $3$ in $G - \{c,d,e \}$ and
  $N(c, G) = \{a,b,d,e \}$, it follows that this vertex $f$ must be in
  the set $V(G) \setminus \{a,b,c,d,e \}$ and that $f$ is adjacent
  both $d$ and $e$. The vertex $f$ has degree $4$ in $G$. It also
  follows from the fact that $G - \{ a,b,c,d,e \}$ contains no edges
  that $f$ must be adjacent to both $a$ and $b$. Since $a$ has degree
  at least $5$ in $G$, it follows that $a$ must be adjacent to some
  vertex $g \in V(G) \setminus \{a,b,c,d,e,f \}$. Then $\col(G - a -
  g) \leq \col(G) - 2 = 3$. On the other hand, $\{ b,c,d,e,f \}$
  induce a subgraph of $G-a-g$ of minimum degree $3$, a
  contradiction. This contradiction implies that $G$ contains no two
  adjacent vertices both of which have degree greater than $4$.
\end{proof}
Recently, the first author obtained a characterisation of what he
called \emph{minimal critical graphs with minimum degree $4$}. It turns out
that our double-$\col$-critical graphs with colouring number $5$ are
such graphs, and so -- using the characterisation of minimal critical
graphs of minimum degree $4$ -- we obtain a characterisation of the
double-$\col$-critical graphs with colouring number $5$.

In the following result, which is the main result of this paper, we let
$Q_1$, $Q_2$, and $Q_3$ denote the graphs depicted in
Figure~\ref{fig:nonSquareCycleDoubleColCritical}.

\begin{theorem}\label{th:doubleCriticalColouringNumber5} A graph is
  double-$\col$-critical with colouring number $5$ if and only if it
  is isomorphic to $Q_1$, $Q_2$, $Q_3$, or the square of a cycle of
  length at least $5$.
\end{theorem}

The graph $Q_2$ is the dual of the Herschel graph which is the smallest nonhamiltonian polyhedral graph.

In order to prove Theorem~\ref{th:doubleCriticalColouringNumber5}, we
first need to introduce a bit of notation and state the abovementioned
characterisation of minimal critical graphs with minimum degree $4$.

For the remaining part of this section we shall be using the following
notation. We shall let $\mc{C}$ denote the set of simple connected
graphs of minimum degree at least $4$. An edge $e$ of a graph $G$ in
${\cal C}$ is {\em essential} if the graph $G-e$ obtained from $G$ by
deleting $e$ is not in ${\cal C}$, and let us call $e$ {\em critical}
if the graph $G/e$ obtained by contracting $e$ and simplifying is not
in ${\cal C}$. An edge $e$ is essential if and only if $e$ is a bridge
or at least one of its endvertices has degree $4$; and $e$ is critical
if and only if the endvertices of $e$ have a common $4$-neighbour or
$N(V(e), G)$ consists of three common neighbours of the endvertices of
$e$. We are now interested in the {\em minimal critical} graphs in
${\cal C}$, that is, graphs $G \in \mc{C}$ with the property that each
edge of $G$ is both essential and critical.

\begin{figure}
  \begin{center} 
\scalebox{0.65}{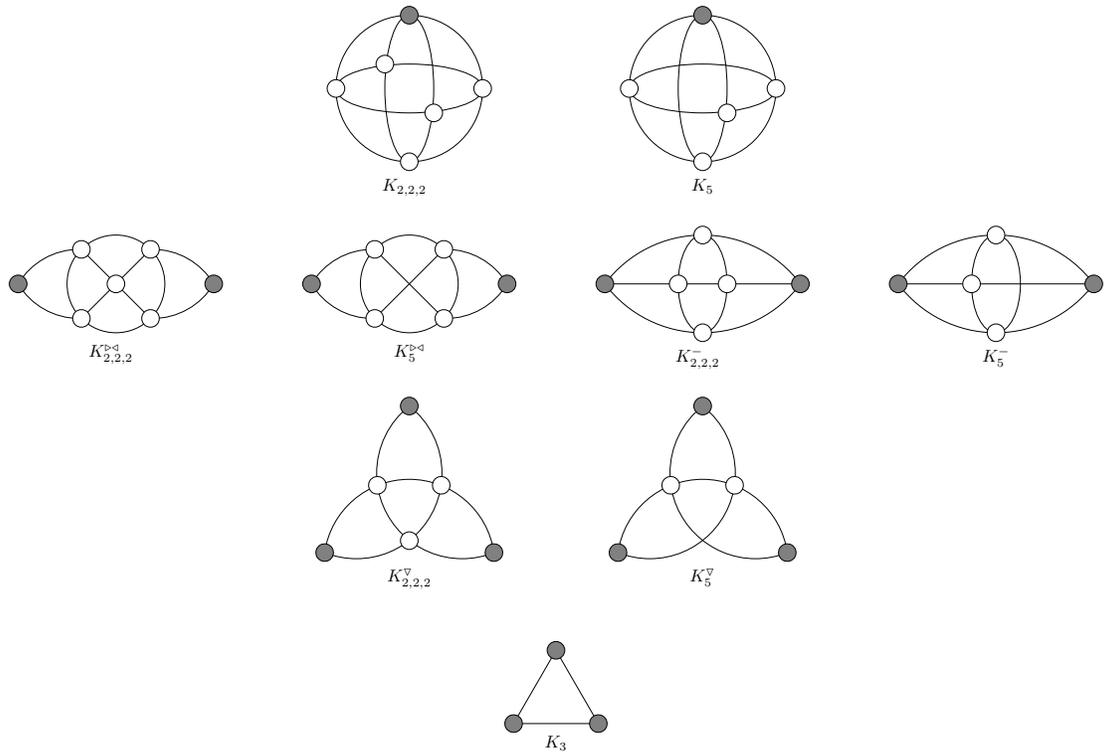}
  \end{center}
  \caption{\label{F1} The nine bricks. Vertices of attachment are displayed solid.}
\end{figure}

For the description of the minimal critical graphs in $\mc{C}$, we
shall consider a number of \emph{bricks}, that is, any graph
isomorphic to one of the following nine graphs: $K_5$, $K_{2,2,2}$,
$K_5^-$, $K_{2,2,2}^-$, $K_5^\triangledown$,
$K_{2,2,2}^\triangledown$, $K_5^{\triangleright \triangleleft}$,
$K_{2,2,2}^{\triangleright \triangleleft}$, or $K_3$ which are
depicted in Figure \ref{F1}. Each brick comes together with its
{\em vertices of attachment}: For $K_5$ and $K_{2,2,2}$, this is an
arbitrary single vertex, for the other seven bricks these are its
vertices of degree less than $4$. The remaining vertices of the brick
are its {\em internal vertices}, and the edges connecting two inner
vertices are called its {\em internal edges}. Observe that every brick
$B$ has one, two, or three vertices of attachment, and that they are
pairwise nonadjacent unless $B$ is the triangle, that is, $K_3$.

It turns out that the minimal critical graphs from ${\cal C}$ are
either squares of cycles of length at least $5$, or they are the edge
disjoint union of bricks, following certain rules. This is made
precise in the following theorem.
\begin{theorem}[Kriesell~\cite{Kriesell:1325227}] \label{T5}
  A graph is a minimal critical graph in ${\cal C}$ if and only if it is the square of a cycle of length at least $5$ or 
  arises from a connected multihypergraph $H$ of minimum degree at least $2$ with at least one edge
  and $|V(e)| \in \{1,2,3\}$ for all hyperedges $e$ 
  by replacing each hyperedge $e$ by a brick $B_e$ (see Figure \ref{F1}) such that the vertices of attachment of $B_e$ are those in $V(e)$ and at the same time the only objects of $B_e$ contained in more than one brick, and

  \begin{itemize}
    \item[\emph{(TB)}] the brick $B_e$ is triangular only if each vertex $x \in V(e)$ is incident with precisely one hyperedge $f_x$ different from $e$ and the corresponding brick $B_{f_x}$ is neither $K_5$, $K_5^-$, $K_{2,2,2}$, nor $K_{2,2,2}^-$, and, for any other vertex $y \in V(e) \setminus \{ x \}$ and hyperedge $f_y$ containing $y$ but distinct from $e$, we have
\begin{itemize}
\item[\emph{(i)}] $V(f_x) \cap V(f_y) \not= \emptyset$ only if not both of $B_{f_x}$ and $B_{f_y}$ are triangular, and 
\item[\emph{(ii)}] $f_x = f_y$ only if $B_{f_x}$ is $K_5^{\triangleright \triangleleft}$ or $K_{2,2,2}^{\triangleright \triangleleft}$.
\end{itemize}
  \end{itemize}
\end{theorem}

\begin{proof}[Proof of Theorem~\ref{th:doubleCriticalColouringNumber5}.]
In order to prove the desired result, we prove the following
equivalent statement. \\

\noindent A graph is double-$\col$-critical with colouring number $5$
if and only if it is the square of a cycle of length at least $5$ or
one of the three graphs obtained by taking the union of two graphs
$G_1$ and $G_2$ such that $G_i \simeq K_5^\triangledown$ or $G_i
\simeq K_{2,2,2}^\triangledown$ for $i \in \{1,2\}$ and $x \in V(G_1)
\cap V(G_2)$ if and only if $x$ has degree $2$ in $G_1$ and degree $2$
in $G_2$. \\

The `if'-part of the statement above is straightforward to verify and
it is left to the reader. \\

Let $G$ denote an arbitrary double-$\col$-critical graph with
colouring number $5$.  It follows from the definition of
double-$\col$-critical graphs,
Observation~\ref{obs:existenceOfColVertexCriticalSubgraph},
Observation~\ref{obs:345770asdfs456dflkj345},
Observation~\ref{obs:23fds45as79}, and
Observation~\ref{obs:345897awlekrf} that $G$ has the following
properties.
\begin{itemize}
  \item[(a)] $G$ has minimum degree $4$;
  \item[(b)] if $x$ and $y$ are adjacent vertices, then at least one of them has degree $4$;
  \item[(c)] if $x$ and $y$ are adjacent vertices, then they have a common $4$-neighbour; and
  \item[(d)] if $x$ and $y$ are adjacent in $G$, then $G- x - y$ has no induced subgraph of minimum degree at least $3$.
\end{itemize}
By (a), $G$ is in $\mc{C}$. By (b), every edge of $G$ is essential,
and, by (c), every edge of $G$ is critical. Thus, $G$ is minimal critical in
${\cal C}$, and so Theorem~\ref{T5} applies. Suppose that $G$ is not
the square of a cycle. Then $G$ has a representation by a multihypergraph $H$ as described in Theorem~\ref{T5}.

If $e$ is a $1$-hyperedge then the unique attachment vertex of the
corresponding brick $B_e$ in $G$ is a cutvertex of $G$, a
contradiction to Proposition~\ref{prop:98ysdfhlkjasdf3}.

Suppose that there exists a $2$-hyperedge $e$ with $V(e)=\{u,v\}$. If
$B_e$ is $K_5^{\triangleright \triangleleft}$ or
$K_{2,2,2}^{\triangleright \triangleleft}$ then $B_e-V(e)$ is an
induced subgraph of $G$ of minimum degree $3$, and since the vertex $u
\in V(e)$ has only two neighbours in that subgraph, there is an edge
in $G - ( V(B_e) \setminus V(e) )$, contradicting (d).  If, otherwise,
$B_e$ is $K_5^-$ or $K_{2,2,2}^-$ then, by Theorem~\ref{T5}, $B_e$ is
an induced subgraph of $G$ of minimum degree $3$. By (a), there is a
vertex in $V(G) \setminus V(B_e)$, and it has a neighbour in $V(G)
\setminus V(B_e)$. This contradicts (d). Hence there are only
$3$-hyperedges in $H$.

 Suppose that $H$ contains a $3$-hyperedge $e$ for which the
 corresponding brick $B_e$ is triangular. It follows from (a) and (d)
 that some vertex $q \in V(G) \setminus V(B_e)$ is adjacent to at
 least two vertices in $V(B_e)$. The vertex $q$ is not adjacent to all
 three vertices of $V(B_e)$, since otherwise $G[V(B_e) \cup \{ q \}]$
 would induce a $4$-clique in $G$ which contradicts
 Observation~\ref{obs:345987asdlhf435asdf}. Let $x$ and $y$ denote the
 neighbours of $q$ in $V(B_e)$. By Theorem~\ref{T5}~(TB), $x$ is
 incident to exactly one hyperedge $f_x$ different from
 $e$. Similarly, $y$ is incident to exactly one hyperedge $f_y$
 different from $e$. If $f_x = f_y$, then, by
 Theorem~\ref{T5}~(TB.ii), $B_{f_x}$ is $K_5^{\triangleright
   \triangleleft}$ or $K_{2,2,2}^{\triangleright \triangleleft}$, in
 particular, $f_x$ is a $2$-hyperedge, a contradiction. Hence $f_x
 \neq f_y$ and so, since $q \in V(f_x) \cap V(f_y)$, it follows from
 Theorem~\ref{T5}~(TB.i) that not both $B_{f_x}$ and $B_{f_y}$ are
 triangular bricks. The fact that $f_x$ and $f_y$ are distinct and $q
 \in V(f_x) \cap V(f_y)$ implies that $q$ is an attachment vertex of
 both $B_{f_x}$ and $B_{f_y}$. Since $q$ is adjacent to $x$ and both
 $q$ and $x$ are attachment vertices, it follows that $B_{f_x}$ must
 be triangular. Similarly, $B_{f_y}$ must be triangular, and so we
 have obtained a contradiction. This shows that each hyperedge in $H$
 is of the type $K_5^\triangledown$ or $K_{2,2,2}^\triangledown$.

Let $e$ denote an arbitrary $3$-hyperedge of $H$. If there are two
vertices $x,y \in V(e)$ of degree exceeding $4$ in $G$ then $G-(V(B_e)
\setminus \{x,y\})$ has minimum degree at least $3$, contradicting (d)
applied to any internal edge of $B_e$.  Therefore, if there is a
vertex $x \in V(e)$ of degree exceeding $4$ in $G$, then the two
vertices $y,z \in V(e) \setminus \{ x \}$ are incident with precisely
one further $3$-hyperedge $f_y$ and $f_z$, respectively, both distinct
from $e$. If $f_y \not= f_z$ then one may argue as above that $G -
V(B_e- x)$ has minimum degree at least $3$, contradicting (d) applied
to any internal edge of $B_e$.  Hence $f_y=f_z=:f$. Let $w$ be the
vertex in $V(f) \setminus \{y,z\}$. If $w \not= x$ then $\{w,x\}$
forms a $2$-separator, and otherwise $w=x$ is a cutvertex as $x$ has
degree exceeding $4$.  In either case, $G- (V(B_e - x) \cup V(B_f))$
has minimum degree at least $3$, again contradicting (d).

Hence all vertices of attachment have degree $4$ in $G$. Let $e$ denote
a $3$-hyperedge in $H$, and let $x$, $y$, and $z$ denote the vertices
of $V(e)$. Again let $f_x,f_y,f_z$ denote the unique $3$-hyperedge
distinct from $e$ incident with $x,y,z$, respectively.  If they are
pairwise distinct then $G-V(B_e)$ has minimum degree at least $3$,
contradiction to (d).  If $f:=f_y=f_z \not= f_x$ then let $w$ be the
vertex in $V(f)$ distinct from $y,z$. As $f_x \not=f$, we have $w
\not= x$, so that $G-(V(e) \cup V(f))$ has minimum degree at least
$3$, contradicting (d), unless there is a vertex in $V(f_x)$ adjacent
to both $x$ and $w$; in this latter case, the unique $3$-hyperedge
distinct from $f$ incident with $w$ must be $f_x$, and so the vertex
$u$ in $V(f_x) \setminus \{w,x\}$ is a cutvertex of $G$, a
contradiction to Proposition~\ref{prop:98ysdfhlkjasdf3}. Hence
$f_x=f_y=f_z$, and the desired statement follows.
\end{proof}
Given our success in characterising the double-$\col$-critical graphs
with colouring number $5$, we venture to ask for a characterisation of
the double-$\col$-critical graphs with colouring number $6$. If $G$ is
a double-$\col$-critical graph, then $G + K_k$ is a
double-$\col$-critical graph with $\col(G + K_k) = \col(G) + k$ (see
Proposition~\ref{prop:completeJoinAndDoubleColCritical}). This implies
that the graphs $Q_1 + K_1$, $Q_2 + K_1$, $Q_3 + K_1$, and $C + K_1$,
where $C$ is the square of any cycle of length at least $5$, are all
double-$\col$-critical graphs with colouring number $6$. These are not
the only double-$\col$-critical graphs with colouring number $6$; the
icosahedral graph is yet another double-$\col$-critical graph with
colouring number $6$. This latter fact was also observed by
Stiebitz~\cite[p.\ 323]{MR1413661}, although in a somewhat different
setting.\label{page:doubleColCritical_col_6} The standard $6$-regular
toroidal graphs obtained from the toroidal grids by adding all
diagonals in the same direction have colouring number $7$ and are
double-$\col$-critical.

\paragraph{Complete joins of double-col-critical graphs.}
In~\cite{KawarabayashiPedersenToftEJC2010}, it was observed that
  if $G$ is the complete join $G_1 + G_2$, then $G$ is
  double-critical if and only if both $G_1$ and $G_2$ are
  double-critical. Next we prove that the `if'-part of the
  analogous statement for double-$\col$-critical graphs is true. The
  `only if'-part is not true, as follows from considering the
  double-$\col$-critical graph $C_ 6^2$: We have $C_6^2 \simeq C_4 +
  \overline{K_2}$ but neither $C_4$ nor $\overline{K_2}$ is
  double-$\col$-critical.
\begin{proposition} \label{prop:completeJoinAndDoubleColCritical}
  If $G_1$ and $G_2$ are two disjoint double-$\col$-critical graphs, then the complete
  join $G_1 + G_2$ is also double-$\col$-critical with
\begin{equation*}
\col(G_1 + G_2) = \min \{ \col(G_1) + n(G_2), \col(G_2) + n(G_1) \} 
\end{equation*}
\end{proposition}
\begin{proof}
  Let $G_1$ and $G_2$ denote two disjoint double-$\col$-critical
  graphs. Then, by Observation~\ref{obs:345770asdfs456dflkj345}, both
  $G_1$ and $G_2$ are $\col$-vertex-critical, and so, by
  Proposition~\ref{prop:completeJoinsAndColouringNumber},
\begin{equation*}
\col(G_1 + G_2) = \min \{ \col(G_1) + n(G_2), \col(G_2) + n(G_1) \} 
\end{equation*}
We need to prove that $\col((G_1 + G_2) - x - y) \leq \col(G_1 + G_2)
- 2$ for every edge $e = xy \in E(G)$; by symmetry, it suffices to
consider (1) $x,y \in V(G_1)$ and (2) $x \in V(G_1)$ and $y \in V(G_2)$. Suppose $x,y \in V(G_1)$. Then
\begin{align*}
\col((G_1 + G_2) - x - y) & = \col( (G_1 - x - y) + G_2 ) \\
& \leq \min \{ \col(G_1 - x - y) + n(G_2), \col(G_2) + n(G_1 - x - y) \} \\
& \leq \min \{ \col(G_1) - 2 + n(G_2), \col(G_2) + n(G_1) - 2   \} \\
& =  \min \{ \col(G_1) + n(G_2), \col(G_2) + n(G_1)  \} - 2
\end{align*}
where we applied
Proposition~\ref{prop:completeJoinsAndColouringNumber} and the fact
that $G_1$ is double-$\col$-critical. A similar argument applies in
case (2). We omit the details. 
\end{proof}
Proposition~\ref{prop:completeJoinAndDoubleColCritical} and the fact
that both $C_6^2$ and $K_t$ are double-$\col$-critical immediately
implies the following result, which, in particular, shows that, for
each integer $k \geq 6$, there is a non-regular double-$\col$-critical
graph with colouring number $k$.
\begin{corollary}
  For any positive integer $t$, the graph $G_t : = C_6^2 + K_t$ is
  a double-$\col$-critical graph with $\col(G_t) = t+5$, $\delta(G_t) =
  n(G_t) - 2$ and $\Delta(G_t) = n(G_t) - 1$.
\end{corollary}
\section{Double-col-critical edges}
In~\cite{KawarabayashiPedersenToftEJC2010}, Kawarabayashi, the second
author, and Toft initiated the study of the number of double-critical
edges in graphs. In this section, we study the number of
double-$\col$-critical edges in graphs. Kawarabayashi, the second
author, and Toft proved the following theorem for which we shall prove
an analogue for the colouring number.

The complete join $C_n + K_1$ of a cycle $C_n$ and a single vertex is
referred to as a \emph{wheel}, and it is denoted $W_n$. If $n$ is odd,
we refer to $W_n$ as an \emph{odd wheel}.
\begin{theorem}[Kawarabayashi, Pedersen \& Toft~\cite{KawarabayashiPedersenToftEJC2010}] \label{th:chi4_and_double_critical_edges}
  If $G$ denotes a $4$-critical non-complete graph, then $G$ contains at
  most $m(G)/2$ double-critical edges. Moreover, $G$ contains
  precisely $m(G)/2$ double-critical edges if and only if $G$ is an odd wheel of order at least $6$.
\end{theorem}
The following result is just a slight reformulation of
Theorem~\ref{th:chi4_and_double_critical_edges}.
\begin{corollary} \label{cor:chi4_and_double_critical_edges}
  If $G$ denotes a $4$-chromatic graph with no $4$-clique, then $G$ contains at
  most $m(G)/2$ double-critical edges. Moreover, $G$ contains
  precisely $m(G)/2$ double-critical edges if and only if $G$ is an odd wheel of order at least $6$.
\end{corollary}
\begin{proof}
  Let $G$ denote a $4$-chromatic graph with no $4$-clique. If $e=xy$
  is a double-critical edge in $G$, then $e$ is a critical edge of $G$
  and $x$ is a critical vertex of $G$. We remove non-critical elements
  from $G$ until we are left with a $4$-critical subgraph $G'$. At no
  point did we remove an endvertex of a double-critical edge. Thus,
  the number of double-critical edges in $G$ is equal to the number of
  double-critical edges in $G'$. Clearly, $G'$ is a non-complete
  graph, and so, by Theorem~\ref{th:chi4_and_double_critical_edges},
  the number of double-critical edges in $G'$ is at most $m(G')/2$
  which is at most $m(G)/2$. The second part of the corollary now
  follows easily.
\end{proof}
The following result --- which is an analogue of
Theorem~\ref{th:chi4_and_double_critical_edges} with the chromatic
number replaced by the colouring number --- extends
Observation~\ref{obs:j3495yfh394tsdf}.

\begin{proposition} \label{prop:col4_and_double_col_critical_edges} If
  $G$ denotes a $4$-$\col$-critical non-complete graph, then $G$
  contains at most $m(G)/2$ double-$\col$-critical edges. Moreover,
  $G$ contains precisely $m(G)/2$ double-$\col$-critical edges if and
  only if $G$ is a wheel of order at least $6$.
\end{proposition}

\begin{lemma} \label{lem:col4_and_double_col_critical_edges}
If $e$ and $f$ are two double-$\col$-critical edges in a
$4$-$\col$-critical non-complete graph, then $e$ and $f$ are incident.
\end{lemma}
\begin{proof}
  Let $G$ denote a $\col$-critical graph with $\col(G) = 4$. Then, by
  Observation~\ref{obs:existenceOfColCriticalSubgraph},
  $\delta(G)=3$. We must have $\omega(G) \leq 3$, since $G$ is a
  $4$-$\col$-critical non-complete graph.

Suppose $e$ is an arbitrary double-$\col$-critical edge in $G$. Then
$\col(G - V(e)) = 2$ which, by
Observation~\ref{obs:3457089asdflkjh345}~(iii), means that $G - V(e)$
is a forest containing at least one edge and, since $\delta(G)=3$,
$\delta(G-V(e))\geq 1$ and each leaf in $G - V(e)$ is adjacent to both
endvertices of $e$ in $G$. Let $u$ and $v$ denote two leafs of $G -
V(e)$. Now, if $G$ contains some double-$\col$-critical edge $f$ which
is not incident to $e$, then $G - V(f)$ contains no cycles, since
$\col(G - V(f)) = 2$, and so $f$ is incident to both $u$ and $v$,
which implies $G[ \{u,v\} \cup V(e)] \simeq K_4$, a
contradiction. This means that any two double-$\col$-critical edges of
$G$ are incident, and the proof is complete.
\end{proof}

\begin{proof}[Proof of Proposition~\ref{prop:col4_and_double_col_critical_edges}] Let $G$
  denote a $4$-$\col$-critical non-complete graph. Then $n(G) \geq 5$
  and, by Observation~\ref{obs:existenceOfColCriticalSubgraph},
  $\delta(G) = 3$ which implies $m(G) \geq \lceil n(G) \cdot \delta(G)
  / 2 \rceil \geq 8$. By
  Lemma~\ref{lem:col4_and_double_col_critical_edges}, we only have to
  consider two cases: (i) $G$ contains three incident
  double-$\col$-critical edges $xy$, $yz$, and $xz$, or (ii) there is
  a vertex $v \in V(G)$ such that every double-$\col$-critical edge of
  $G$ is incident to $v$. If (i) holds, then, since $m(G) \geq 8$, the
  desired statement follows. Suppose (ii) holds. Then the number of
  double-$\col$-critical edges in $G$ is at most $\deg(v, G)$. We may
  assume that there is at least one double-$\col$-critical edge, say,
  $vw$ in $G$. Suppose $G - v$ is disconnected. Then, since
  $\delta(G)=3$, each component of $G - v$ has minimum degree at least
  $2$, and so, in particular, some component of $G- v - w$ has minimum
  degree at least $2$. This, however, contradicts the fact that $G - v
  - w$ is a forest. Hence $G - v$ is connected. By
  Observation~\ref{obs:ColDropsByOneExactly}, $\col(G - v) \geq 3$ and
  so, by Observation~\ref{obs:3457089asdflkjh345}~(iii), $G - v$
  contains a cycle. Hence $G-v$ is a connected graph with at least one
  cycle, and so $m(G-v) \geq n(G-v)$. Thus,
$$m(G) = \deg(v, G) + m(G-v) \geq \deg(v, G) + (n(G) - 1) \geq 2\deg(v,G)$$
which implies that the number of double-$\col$-critical edges is at
most $m(G)/2$ and that the number of double-$\col$-critical edges is equal
to $m(G)/2$ only if $\deg(v,G) = n(G)-1$ and $G-v$ is a cycle.

Conversely, if $G$ is a wheel on at least five vertices, then it is
easy to see that exactly $m(G)/2$ edges of $G$ are
double-$\col$-critical.  This completes the proof.
\end{proof}
Let $k$ denote some integer greater than $3$. Let $D_k$ denote the $2k$-cycle
with vertices labelled cyclically $v_0v_1\dots v_ku_{k-1}u_{k-2}\dots
u_1$. Let $F_k$ denote the graph
$$D_k^2 - u_1v_1 - u_{k-1}v_{k-1} + v_1v_{k-1} + u_1u_{k-1}$$ 
Figure~\ref{fig:F_5} depicts a drawing of $F_5$.
\begin{figure}
  \begin{center} 
\scalebox{0.6}{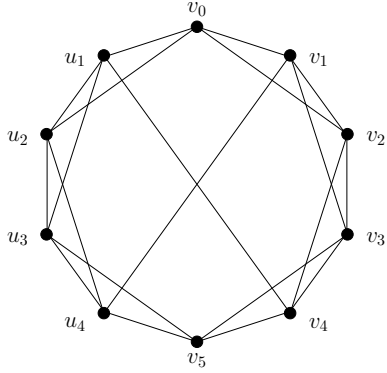}
  \end{center}
  \caption{The graph $F_5$ has colouring number $5$ and all edges of $F_5$, except $u_1v_4$ and $v_1u_4$, are double-$\col$-critical.}
 \label{fig:F_5}
\end{figure}
\begin{observation} \label{obs:F_k}
For every integer $k$ greater than $3$, the graph $F_k$, as defined as
above, is a $5$-$\col$-critical graph with colouring number $5$ in
which all edges except $v_1v_{k-1}$ and $u_1 u_{k-1}$ are
double-$\col$-critical.
\end{observation}
\begin{proposition} \label{prop:G_p_k}
  For each integer $p$ greater than $4$ and positive real number
  $\epsilon$, there is a $p$-$\col$-critical graph $G$ with the ratio
  of double-$\col$-critical edges between $1- \epsilon$ and $1$.
\end{proposition}
\begin{proof}
If $p=5$, the desired result follows directly from
Observation~\ref{obs:F_k} by letting $k$ tend to infinity. Let $p$
denote an integer greater than $5$ and $\epsilon$ a positive real
number. Let $k$ denote an integer a lot greater than $p$, and let $G$
denote the graph obtained by taking the complete join of $F_k$ and
$\overline{K_{p-5}}$. Then, by
Proposition~\ref{prop:completeJoinsAndColCritical}~(ii), $G$ is
$p$-$\col$-critical, and, since $k$ is a lot greater than $p$, all but
the edges $v_1v_{k-1}$ and $u_1 u_{k-1}$ are double-$\col$-critical in
$G$. By letting $k$ tend to infinity the ratio of
double-$\col$-critical edges in $G$ will from a certain point
onwards be between $1 - \epsilon$ and $1$. This completes the
argument.
\end{proof}
Proposition~\ref{prop:G_p_k} means that there is no result
corresponding to
Proposition~\ref{prop:col4_and_double_col_critical_edges} for
colouring numbers greater than $4$.
\section{Concluding remarks}
By a theorem of Mozhan~\cite{MR995391} and, independently,
Stiebitz~\cite{MR882614}, $K_5$ is the only double-critical graph with
chromatic number $5$, that is, $K_5$ is the only $5$-chromatic graph
with $100$\% double-critical edges, but we do not know whether there
are non-complete $5$-chromatic graphs with the percentage of
double-critical edges arbitrarily close to a
$100$. In~\cite{KawarabayashiPedersenToftEJC2010}, Kawarabayashi, the
second author, and Toft conjectured that if $G$ is a $5$-critical
non-complete graph, then $G$ contains at most $(2 + \frac{1}{3n(G) -
  5}) \frac{m(G)}{3}$ double-critical edges.

As we have seen, the story is a bit different for the colouring number. By
Proposition~\ref{prop:G_p_k} for $p=5$, there are non-complete graphs
with colouring number $5$ with the percentage of
double-$\col$-critical edges arbitrarily close to a $100$. This only
makes Theorem~\ref{th:doubleCriticalColouringNumber5} all the more
interesting. By Theorem~\ref{th:doubleCriticalColouringNumber5}, we
are able to distinguish between graphs with colouring number $5$
having $99.99$\% double-$\col$-critical edges and those that have a
$100$\% double-$\col$-critical edges.

The problem of obtaining a concise structural description of the
double-$\col$-critical graphs with colouring number $k \geq 6$ remains
open. Given any graph $G$, it can be decided in polynomial time
whether or not $G$ is double-$\col$-critical, since the colouring
number itself can be computed in polynomial time. Nevertheless, given
the structural complexity of the double-$\col$-critical graphs with
colouring number $6$ mentioned on
page~\pageref{page:doubleColCritical_col_6}, it seems likely that even
the problem of obtaining a concise structural description of the
double-$\col$-critical graphs with colouring number $6$ is
non-trivial.

\section*{Acknowledgement}
We thank Bjarne Toft for posing the problem of characterising the
double-$\col$-critical graphs and his many insightful comments on colouring and degeneracy of graphs.

\bibliographystyle{plainnat} \bibliography{double-col}

\begin{thebibliography}{15}
\providecommand{\natexlab}[1]{#1}
\providecommand{\url}[1]{\texttt{#1}}
\expandafter\ifx\csname urlstyle\endcsname\relax
  \providecommand{\doi}[1]{doi: #1}\else
  \providecommand{\doi}{doi: \begingroup \urlstyle{rm}\Url}\fi

\bibitem[Bondy and Murty(2008)]{BondyAndMurty2008}
J.~A. Bondy and U.~S.~R. Murty.
\newblock \emph{Graph theory}, volume 244 of \emph{Graduate Texts in
  Mathematics}.
\newblock Springer, New York, 2008.

\bibitem[Erd\H{o}s(1968)]{TihanyProblem2}
P.~Erd\H{o}s.
\newblock Problem 2.
\newblock In \emph{Theory of Graphs (Proc. Colloq., Tihany, 1966)}, page 361.
  Academic Press, New York, 1968.

\bibitem[Erd{\H{o}}s and Hajnal(1966)]{MR0193025}
P.~Erd{\H{o}}s and A.~Hajnal.
\newblock On chromatic number of graphs and set-systems.
\newblock \emph{Acta Math. Acad. Sci. Hungar}, 17:\penalty0 61--99, 1966.

\bibitem[Jensen and Toft(1995)]{JensenToft95}
T.~R. Jensen and B.~Toft.
\newblock \emph{Graph coloring problems}.
\newblock Wiley-Interscience Series in Discrete Mathematics and Optimization.
  John Wiley \& Sons Inc., New York, 1995.
\newblock A Wiley-Interscience Publication.

\bibitem[Kawarabayashi et~al.(2010)Kawarabayashi, Pedersen, and
  Toft]{KawarabayashiPedersenToftEJC2010}
K.~Kawarabayashi, A.~S. Pedersen, and B.~Toft.
\newblock Double-critical graphs and complete minors.
\newblock \emph{Electron. J. Combin.}, 17\penalty0 (1):\penalty0 Research Paper
  87, 27 pp., 2010.

\bibitem[Kriesell(2002)]{MR1892444}
M.~Kriesell.
\newblock A survey on contractible edges in graphs of a prescribed vertex
  connectivity.
\newblock \emph{Graphs Combin.}, 18\penalty0 (1):\penalty0 1--30, 2002.

\bibitem[Kriesell(2011)]{Kriesell:1325227}
M.~Kriesell.
\newblock Nonseparating {$K_4$}-subdivisions in graphs of minimum degree at
  least $4$.
\newblock Technical Report arXiv:1101.5278, Jan 2011.
\newblock IMADA-preprint, 25 pages.

\bibitem[Lick and White(1970)]{MR0266812}
D.~R. Lick and A.~T. White.
\newblock {$k$}-degenerate graphs.
\newblock \emph{Canad. J. Math.}, 22:\penalty0 1082--1096, 1970.

\bibitem[Matula and Beck(1983)]{MR709826}
D.~W. Matula and L.~L. Beck.
\newblock Smallest-last ordering and clustering and graph coloring algorithms.
\newblock \emph{J. Assoc. Comput. Mach.}, 30\penalty0 (3):\penalty0 417--427,
  1983.

\bibitem[Mozhan(1987)]{MR995391}
N.~N. Mozhan.
\newblock Twice critical graphs with chromatic number five.
\newblock \emph{Metody Diskret. Analiz.}, \penalty0 (46):\penalty0 50--59, 73,
  1987.

\bibitem[Pedersen(2010)]{ASP_PhD_Thesis}
A.~S. Pedersen.
\newblock \emph{Contributions to the Theory of Colourings, Graph Minors, and
  Independent Sets}.
\newblock PhD thesis, University of Southern Denmark, 2010.

\bibitem[Stiebitz(1987)]{MR882614}
M.~Stiebitz.
\newblock {$K\sb 5$} is the only double-critical {$5$}-chromatic graph.
\newblock \emph{Discrete Math.}, 64\penalty0 (1):\penalty0 91--93, 1987.

\bibitem[Stiebitz(1988)]{MR1221590}
M.~Stiebitz.
\newblock On {$k$}-critical {$n$}-chromatic graphs.
\newblock In \emph{Combinatorics (Eger, 1987)}, volume~52 of \emph{Colloq.
  {M}ath. {S}oc. {J}\'anos {B}olyai}, pages 509--514. North-Holland, Amsterdam,
  1988.

\bibitem[Stiebitz(1996)]{MR1413661}
M.~Stiebitz.
\newblock Decomposing graphs under degree constraints.
\newblock \emph{J. Graph Theory}, 23\penalty0 (3):\penalty0 321--324, 1996.

\bibitem[Toft(1995)]{MR1373659}
B.~Toft.
\newblock Colouring, stable sets and perfect graphs.
\newblock In \emph{Handbook of combinatorics, {V}ol.\ 1,\ 2}, pages 233--288.
  Elsevier, Amsterdam, 1995.

\end{thebibliography}
\end{document}